%% file: Zloops2_nc.tex
\documentclass[12pt]{amsart}

\usepackage{amsfonts,amsmath,amssymb,amsthm,graphicx,enumerate, mathabx}
\usepackage[english]{babel}
\usepackage[utf8]{inputenc}
\usepackage[T1]{fontenc}
\usepackage{framed}
\usepackage{marginnote,color}
\usepackage{array}
\usepackage{calrsfs}
\usepackage{nicefrac}
\usepackage{geometry}
\usepackage{graphicx,graphics}

\definecolor{myblue}{rgb}{0.09,0.32,0.44} 
\usepackage[pdfusetitle,
  bookmarks=true,bookmarksnumbered=false,bookmarksopen=false,
  breaklinks=false,pdfborder={0 0 0},backref=false,colorlinks=true,
  pdfborderstyle={},linkcolor=myblue,citecolor=black,urlcolor=blue]
 {hyperref}

\edef\sectionsign{\S}

\input{commandes}

\newcommand{\prob}{\mathbb{P}}
\DeclareMathOperator{\Tri}{Tri}

\newcommand{\npath}{n_{\rm paths}}

\newtheorem{theorem}{Theorem}
\newtheorem{lemma}{Lemma}

\newtheorem*{remark*}{Remark}

\theoremstyle{remark}
\newtheorem{definition}{Definition}

\newtheorem*{conjecture*}{Conjecture}

\author[N. Curien]{Nicolas Curien}
\address{Laboratoire de math\'ematique, Universit\'e Paris-Sud, 91400 Orsay, France}
\email{nicolas.curien@gmail.com}

\author[G. Kozma]{Gady Kozma}
\address{Weizmann institute of Science, 76100, Rehovot, Israel}
\email{gady.kozma@weizmann.ac.il}

\author[V. Sidoravicius]{Vladas Sidoravicius}
\address{Courant Institute of Mathematical Sciences, New York\\
NYU-ECNU Institute of Mathematical Sciences at NYU Shanghai\\
Cemaden, S\~ao Jos\'e dos Campos.}
\email{vs1138@nyu.edu}

\author[L. Tournier]{Laurent Tournier}
\address{LAGA, Universit\'e Paris 13, Sorbonne Paris Cit\'e, CNRS, UMR 7539, 93430 Villetaneuse, France. 
 }
\email{tournier@math.univ-paris13.fr}

\keywords{}
\subjclass[2010]{primary ; secondary }

\thanks{The research of G.K. was supported by the Israel Science Foundation.}

\title{}

\newcommand{\Li}{\ldbrack} 
\newcommand{\Ri}{\rdbrack}

\begin{document}
\begin{center}
\Large \textbf{ \textsc{Uniqueness of the infinite}}  \includegraphics[height=1.5cm]{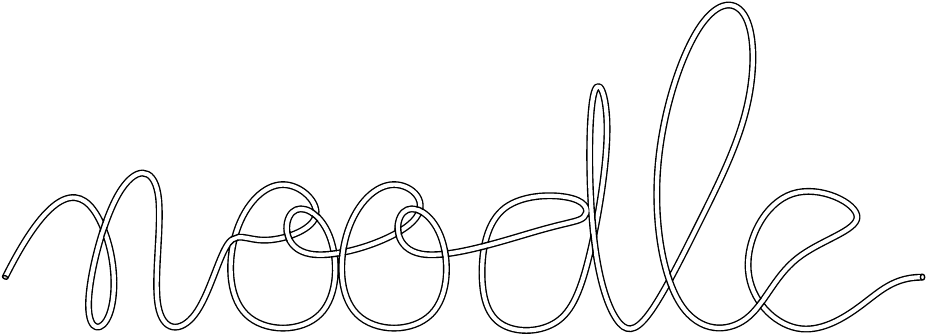}
\end{center}
\maketitle

{\footnotesize \noindent{\slshape\bfseries Abstract.} 
Consider the graph obtained by superposition of an independent pair of uniform infinite non-crossing perfect matchings of the set of integers. We prove that this graph contains at most one infinite path. Several motivations are discussed.} 

\begin{figure}[!h]
 \begin{center}
 \includegraphics[width=5cm]{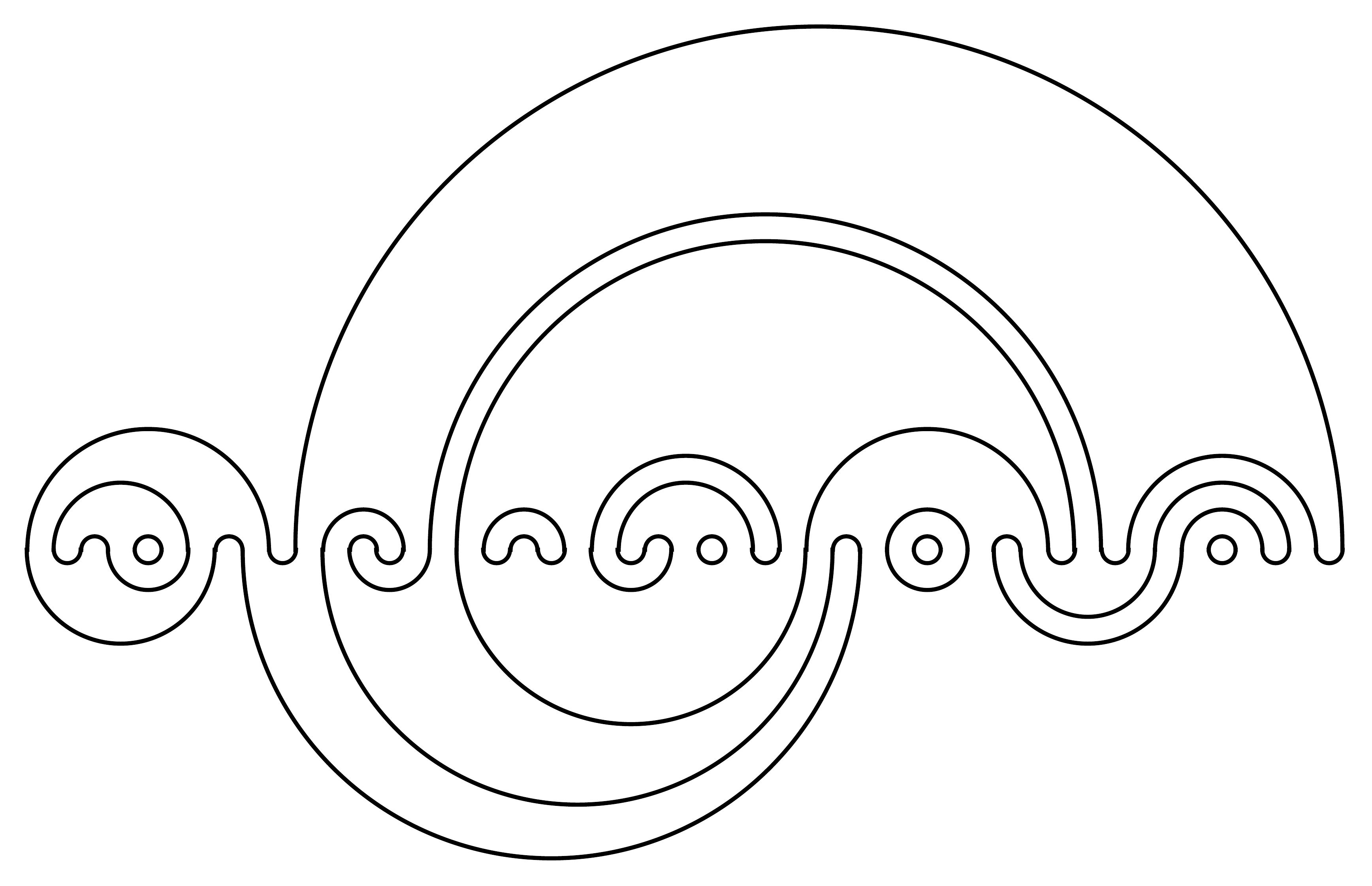} \hspace{1cm}
  \includegraphics[width=5cm]{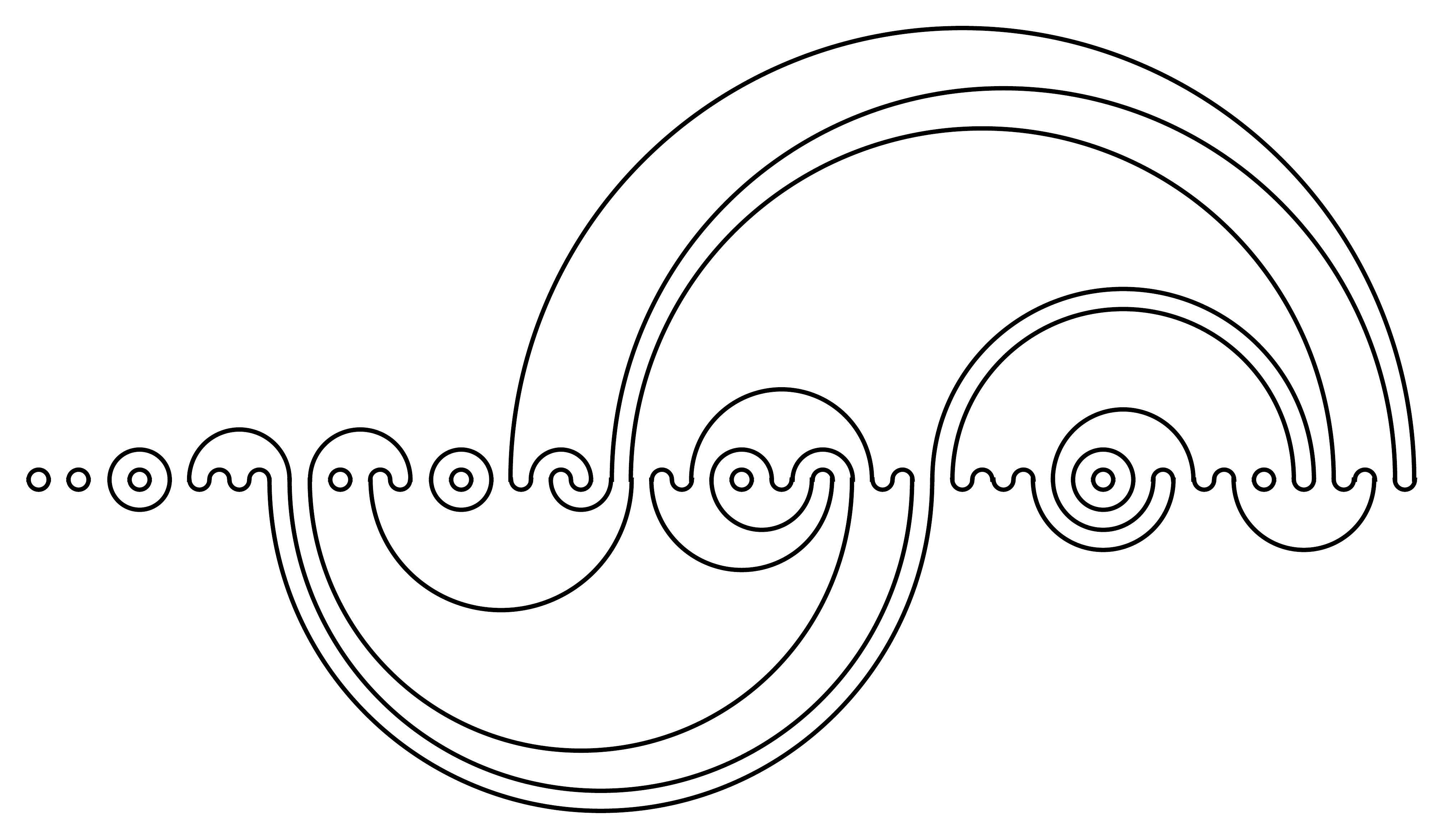} \hspace{1cm}
  \includegraphics[width=5cm]{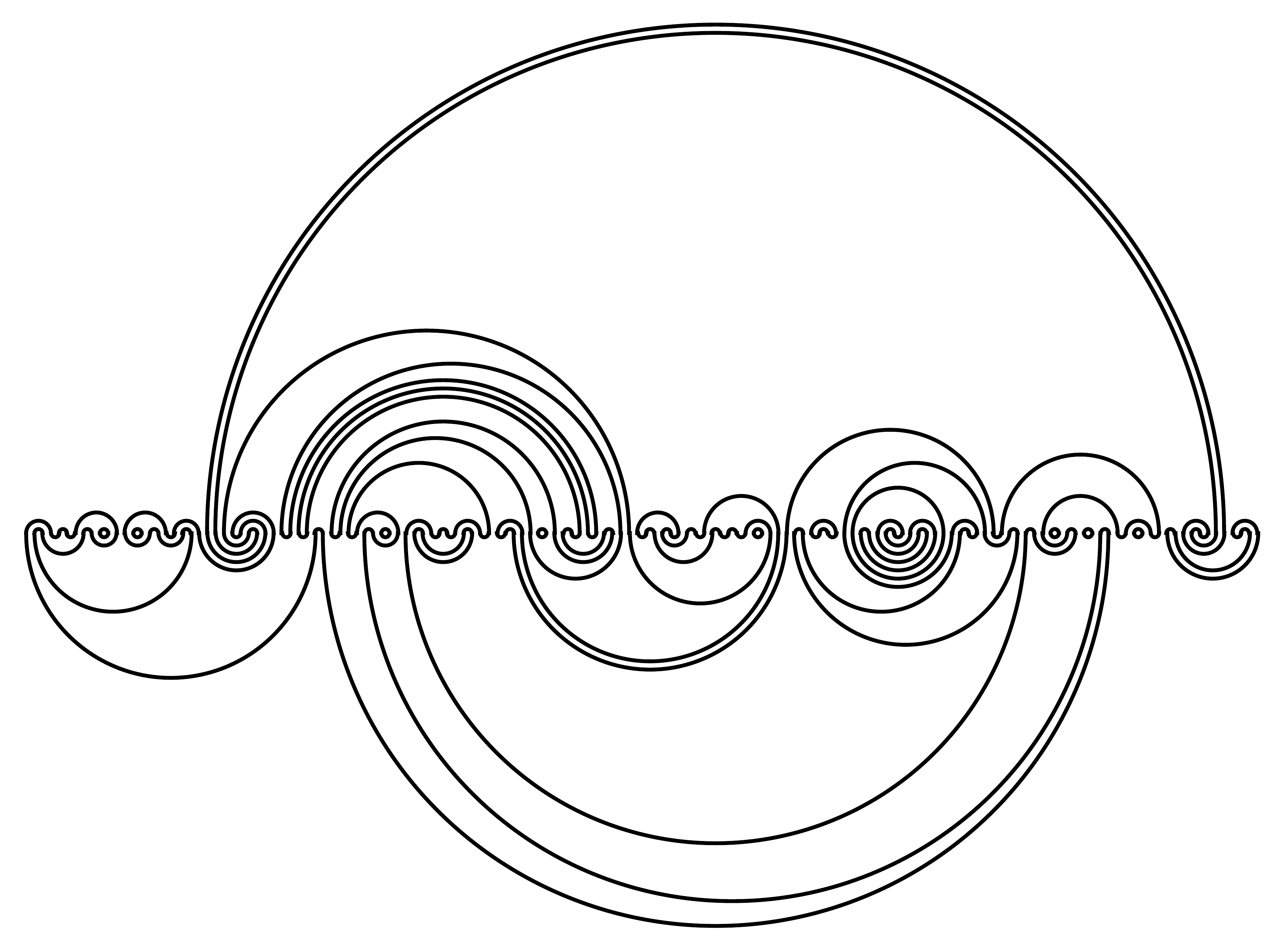}
 \caption{ \label{fig:intro}	The graph obtained by the superposition of two independent planar perfect matchings on respectively 50, 70 and 150 points. Our main result concerns infinite clusters in the infinite version of this construction.}
 \end{center}
 \end{figure}

\section{Introduction}
   \label{s:intro}
   Let us start by defining the model. On the integer line $ \mathbb{Z}$ we sample independently for each site $x \in \mathbb{Z}$ two independent variables $\omega_{x}^{+}$ and $\omega_{x}^{-}$ uniformly in the set $\{+1,-1\}$. We interpret $\omega_{x}^{+}$  (resp.~$\omega_{x}^{-}$) as being a parenthesis: an opening parenthesis for $+1$ and closing one for $-1$ living on the upper half-plane (resp.~lower half-plane). By pairing the parentheses, it is then standard that $(\omega_{x}^{+})_{ x \in \mathbb{Z}}$ yields a perfect matching on $ \mathbb{Z}$, i.e.~an involution  of $ \mathbb{Z}$ without fixed points. This matching is furthermore planar, meaning that we can draw arcs between paired points so that the arcs are non crossing. In most of the following drawings, these arcs will be tents or semi-circles. We  repeat this construction twice, once with the parentheses of the upper half plane, and once in lower half plane. After gluing the top and bottom arches on points of $ \mathbb{Z}$ we are left with a random (multi)graph $G$ with $\npath \in \{0,1,2,\ldots\} \cup \{\infty\}$ many infinite clusters (actually bi-infinite paths).
   
   \begin{theorem} Either $\npath =0$ almost surely or $\npath=1$ almost surely.
\end{theorem}

Unfortunately, we have not been able to decide which of the two alternatives actually holds and leave this as an open question. Before moving to the proof, let us present a few motivations for studying this model, except of course of its intrinsic beauty, see Figures~\ref{fig:intro} and~\ref{fig:one_cycle}.  \medskip

\begin{figure}[]
 \begin{center}
 \includegraphics[width=5cm,angle=-90]{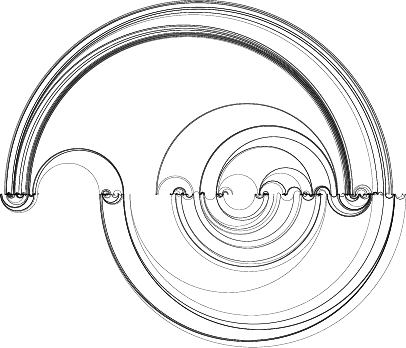}
 \caption{\label{fig:one_cycle}	The cluster of 0 in a typical simulation of the infinite model. This cycle contains 2936 sites.}
 \end{center}
 \end{figure}

\paragraph{\textsc{Random planar maps and Liouville quantum gravity}.}  In the last decades, the geometry of random uniform planar graphs (or maps) has been studied intensively and is now quite well understood \cite{LeGallICM}. The large scale structure of \emph{decorated} planar maps is much less understood and virtually nothing is rigorously known on the asymptotic geometry of these objects. One of the simplest model is that of triangulations given with a spanning tree. Upon cutting along the spanning tree, such a triangulation can be seen as a binary tree of triangles and an ``outside'' planar matching. If we further blow this tree of triangles we end up with a discrete cycle with two systems of non-crossing arches (both counted by Catalan numbers), an outside one identifying pairs of edges and an inside one connecting points. See Figure~\ref{fig:trig}. 

\begin{figure}[!h]
 \begin{center}
 \includegraphics[width=15cm]{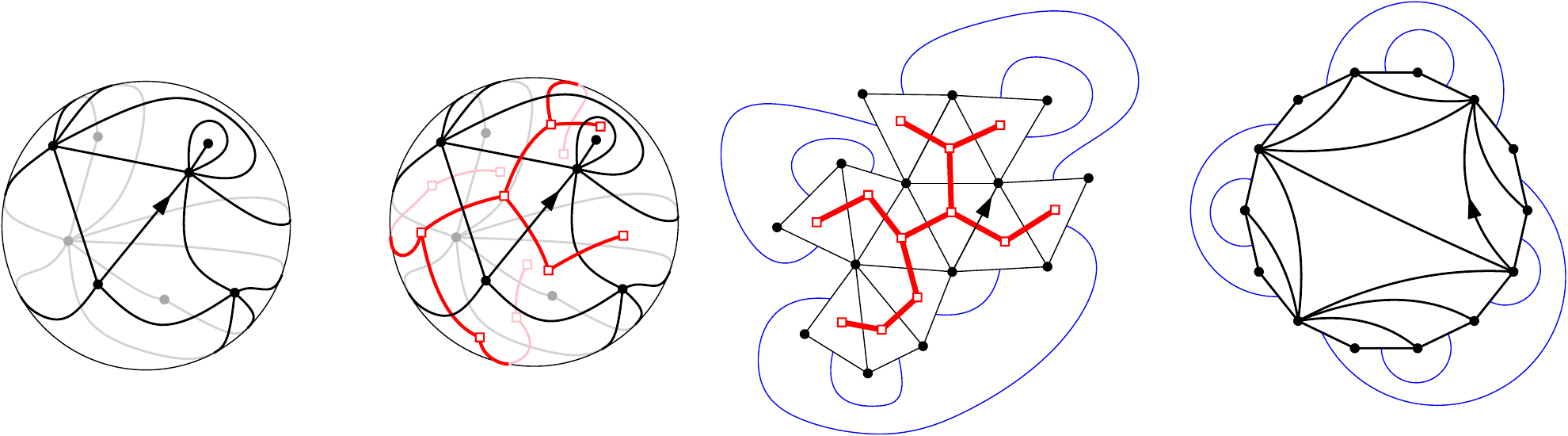}

 \caption{ \label{fig:trig} A triangulation with a spanning tree, cut along its spanning tree and seen as a cycle with two systems of arches.}
 \end{center}
 \end{figure}
 
 Hence our model can be seen as a (infinite) simplified version of the above construction where the two systems of arches play a symmetric role. We hope that the insight we get in studying our model will be useful to understand the geometry of tree-decorated maps. In the continuous setting, the idea  to glue a pair of random trees (which can be seen as continuous limit of planar perfect matchings) appears in the construction of the Brownian map \cite{MM06,LG07} and in Liouville quantum gravity \cite{SheHC,DMSgluingoftrees}.

\medskip
\paragraph{\textsc{Meanders}.} A meander is a self-avoiding closed loop crossing a horizontal line, seen up to topological equivalence, i.e.\ up to a homeomorphism of the plane preserving the horizontal line. In our model, this correspond to gluing two (finite) planar perfect matchings so that the resulting graph $G$ is connected, in other words, the finite clusters in our graph $G$ are meanders. The (asymptotic) enumeration of meanders is a notorious difficult open problem in combinatorics and in theoretical physics, see \cite{GuiHDR}. Our model in contrast allows the explicit computation of several quantities; let us mention for instance that the expected number of circles, i.e.~length 2 clusters, going around or passing through a given vertex, turns out to have the surprisingly simple value $\frac1{2\pi}$. 

\begin{figure}[!h]
 \begin{center}
 \includegraphics[width=4cm]{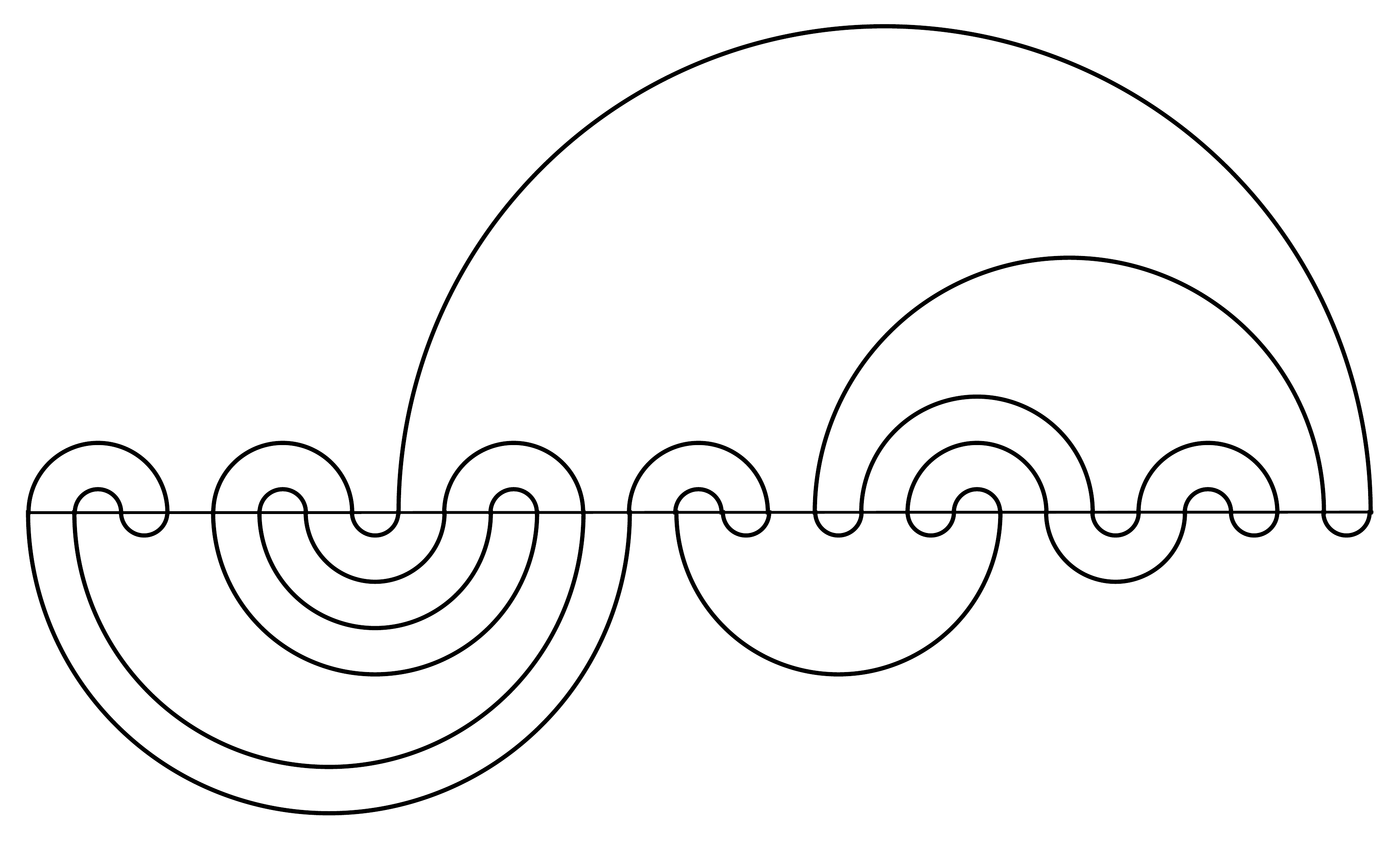} \hspace{2cm}
  \includegraphics[width=4cm]{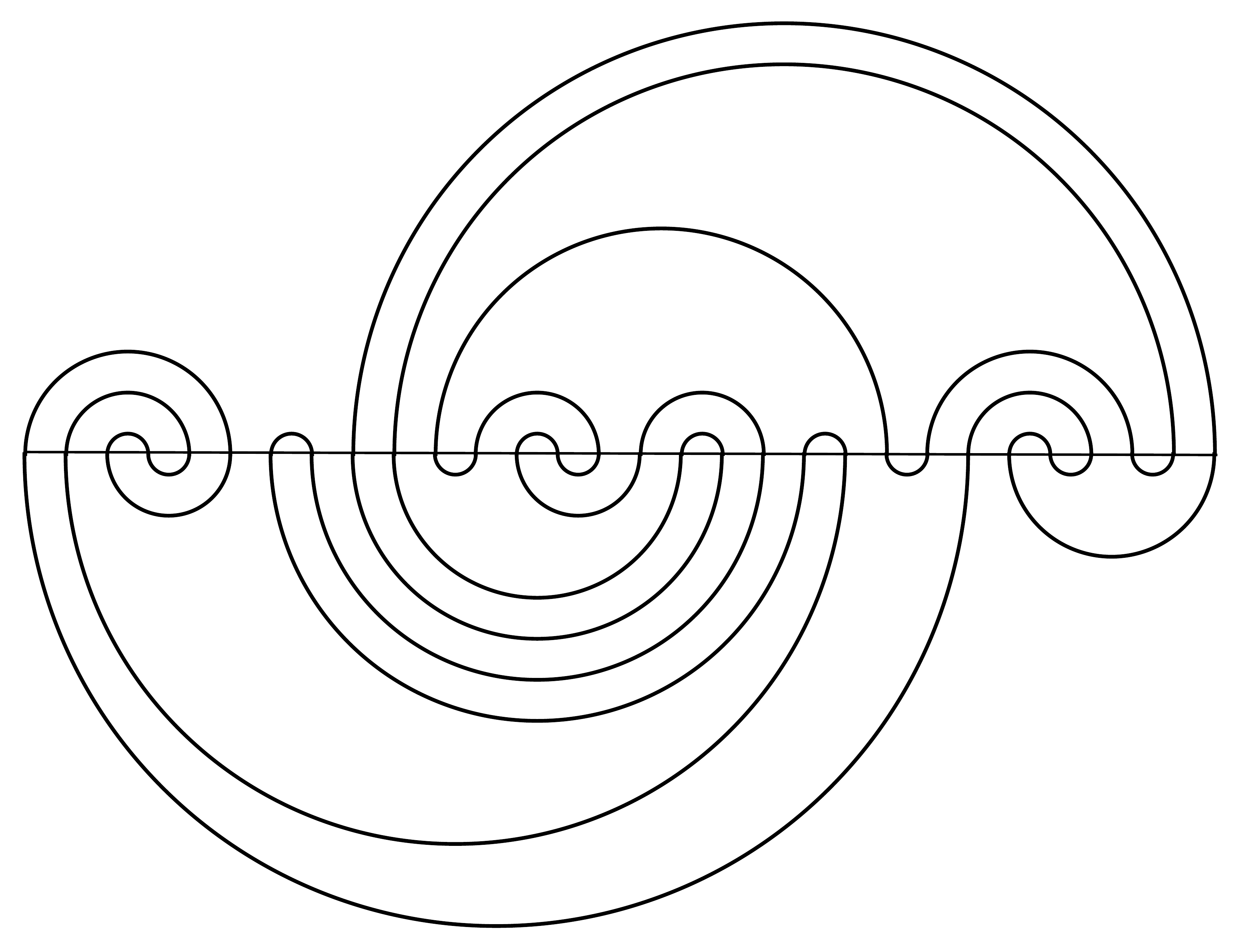}
 \caption{Two random meanders over $30$ vertices.}
 \end{center}
 \end{figure}

\paragraph{\textsc{Proof sketch}.}The proof resembles the Burton-Keane argument for percolation. We first note that $\npath$ is constant a.s.\ due to ergodicity. We then discard the case $\npath > 2$ (infinity included) by a trifurcation argument. The argument does not apply to the paths themselves, of course, but to the space between them (here planarity is crucial). This argument is presented in \sectionsign \ref{sec:trifu}. The argument that precludes the case $\npath=2$ is a local modification argument, but is complicated by the fact that, in fact, a local modification cannot change the number of infinite paths. It can, however, wire them differently i.e.\ make a new path by joining the tails of two existing paths (after some parity issues are resolved). In \sectionsign\ref{sec:rewire} we explain why this is enough.

\medskip
\noindent As already stated, we believe that $\npath=0$. In fact, we believe this holds in much greater generality, here is the precise formulation:
\begin{conjecture*}Let $F^+$ be a random matching of $\Z$ (not necessarily planar), stationary and ergodic with respect to the action of translation by $x$, for every $x\in\Z$. Let $F^-$ be a second matching of $\Z$, also stationary and ergodic to the action of all translations. Assume $F^+$ and $F^-$ are independent (but not necessarily identically distributed). Let $G$ be the random graph whose vertex set is $\Z$ and edge set is the union of $F^+$ and $F^-$. Then $G$ does not have a \emph{unique} infinite cluster, almost surely.
\end{conjecture*}

\section{Finite configurations}
\label{sub:finite_config} Let $\Omega=(\{-1,+1\}^2)^\Z$ be our set of configurations, and let us consider a random variable $\omega=(\omega^+,\omega^-)$ in $\Omega$, where $\omega^+=(\omega^+_n)_{n\in\Z}$ and $\omega^-=(\omega^-_n)_{n\in\Z}$ are independent sequences of i.i.d.~random variables with uniform distribution in $\{-1,+1\}$. We denote by $E^{+}$ (resp.~$E^{-}$) the edges belonging to the upper (resp.~lower) planar perfect matching. For $x \in \mathbb{Z}$ we write  $ \mathcal{C}(x)$ for the cluster of $G$ containing $x$.  
Knowing the restriction of the configuration $\omega$ to a finite interval amounts to knowing a finite sub-graph of $G$ plus the orientation of edges exiting the inverval (i.e., the fact that they leave the interval toward the left or right). Let us introduce some notation and properties regarding this situation, which will be useful in order to make ``local'' modifications of $\omega$. Note indeed that changing $\omega$ in such a finite interval may have global consequences on $G$, unless the location of the out-going edges is preserved. 
Thanks to translation invariance, we consider here integer intervals of the form $\Li1,N\Ri$, where $N\in\N$, without loss of generality. 

\newcommand{\ntot}{n_{\rm tot}}
For $S\subset\Z$, we shall denote the set of configurations on $S$ by
\begin{equation}\label{def_omegaS}
\Omega_S=(\{-1,+1\}^2)^S.
\end{equation}
Let $N\in\N$, and $\eta\in\Omega_{\Li1,N\Ri}$. Let us complete $\eta$ into $\omega\in\Omega$ by letting $\omega_n=(+1,+1)$ for all $n<1$ and $\omega_n=(-1,-1)$ for all $n>N$; we may then define the graph $G=G(\omega)$ as before, and define the numbers of edges outgoing $\Li1,N\Ri$ through the top left and top right (see Figure~\ref{fig:ends_circle_in}):
\begin{align*}
n^+_L = n^+_L(\eta) & =\#\{x\in\Li1,N\Ri\st \exists y<1,\ \{x,y\}\in E^+\},\\
n^+_R = n^+_R(\eta) & =\#\{x\in\Li1,N\Ri\st \exists y>N,\ \{x,y\}\in E^+\},
\end{align*}
and similarly $n^-_L$ and $n^-_R$ using $E^-$, for the lower part, and the total number of boundary edges:
\[\ntot=\ntot(\eta)=n^+_L+n^+_R+n^-_L+n^-_R.\] 
We shall sometimes call these edges ``dangling ends'', or ends, of $\Li1,N\Ri$. These edges are non-crossing and all of them start inside and end outside the circle of diameter $[\frac12,N+\frac12]$ (for any disjoint embedding). Thus, they have a natural cyclic ordering, given by the order on this circle of their last intersection points (or of their \emph{only} intersection, for the embeddings mentioned in the beginning). Let us number them from $1$ to $\ntot$ starting for instance from the bottommost top left outgoing edge and following the clockwise order. By associating the two ends of each connected component in $G$ to each other, $\eta$ defines a non-crossing matching $\sigma$ of $\Li1,\ntot\Ri$; we shall say that $\eta$ \emph{realises} $\sigma$.  
\begin{figure}
\includegraphics[height=5cm]{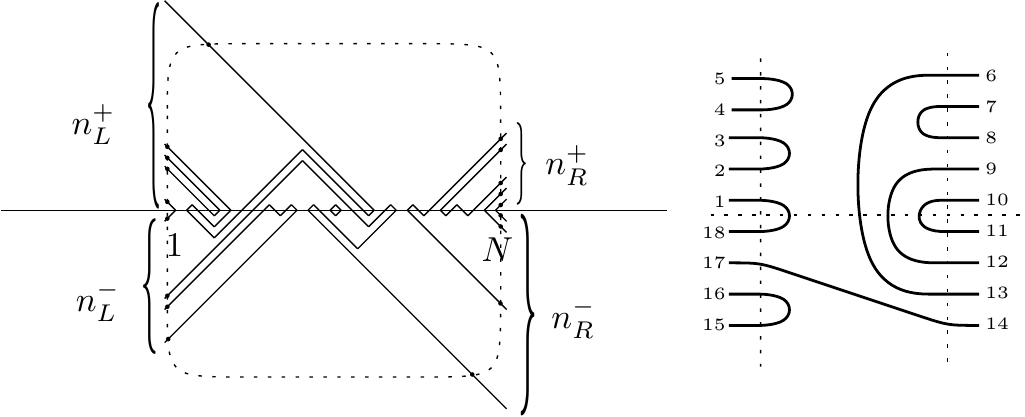}
\caption{Left: Notations for numbers of ends of paths going out of $\Li1,N\Ri$. These ends have a natural cyclic ordering viewed as points of the dotted loop. Right: Non-crossing matching realised by the left configuration.}
\label{fig:ends_circle_in}
\end{figure}

Conversely, any non-crossing matching can be realised with any prescribed numbers of ends provided some room is allowed and the necessary parity conditions hold: 

\begin{lemma}\label{lem:local_config}
Let $N\ge1$ be an integer. For any nonnegative integers $a^+,b^+,a^-,b^-$ such that $a^++b^+$, $a^-+b^-$ and $N$ have same parity, and such that $\ntot\defeq a^++b^++a^-+b^-\le N$, and for any non-crossing matching $\sigma:\Li1,\ntot\Ri\to\Li1,\ntot\Ri$, there exists a configuration $\eta\in\Omega_{\Li 1,N\Ri}$ such that $(n^+_L,n^+_R,n^-_L,n^-_R)(\eta)=(a^+,b^+,a^-,b^-)$ and that realises $\sigma$. 
\end{lemma}

Note that the parity assumptions are necessary conditions for the conclusion to hold: since the $N$ vertices, together with the $n^+_L(\eta)+n^+_R(\eta)$ upper ends are paired by $\eta^+$, their total number $N+n^+_L(\eta)+n^+_R(\eta)$ has to be even, and similarly for the lower ends. 

\begin{proof}
Let us first construct  a configuration that realises $\sigma$ using at most two vertices for each matched pair, disregarding the value of $N$. We will now describe the construction verbally, but the reader is probably better served by simply checking Figure~\ref{fig:exemple_matching}. We realise every horizontal path (i.e.\ one that starts on the left and ends on the right) with both ends above the line by a $\land$ shape, every horizontal path with both ends below the line by a $\lor$ shape, and horizontal below-to-above and above-to-below (only one kind may exist, by planarity), by diagonals. Each 'V' shape uses two vertices, and each diagonal one. Paths starting and ending on the same side (left or right) are then inserted between these, in order of containment, with above-to-above and below-to-below inserted as diagonal strips (each using two vertices) and above-to-below as a $>$ or $<$ (each using one vertex).

\begin{figure}
\includegraphics[height=6cm]{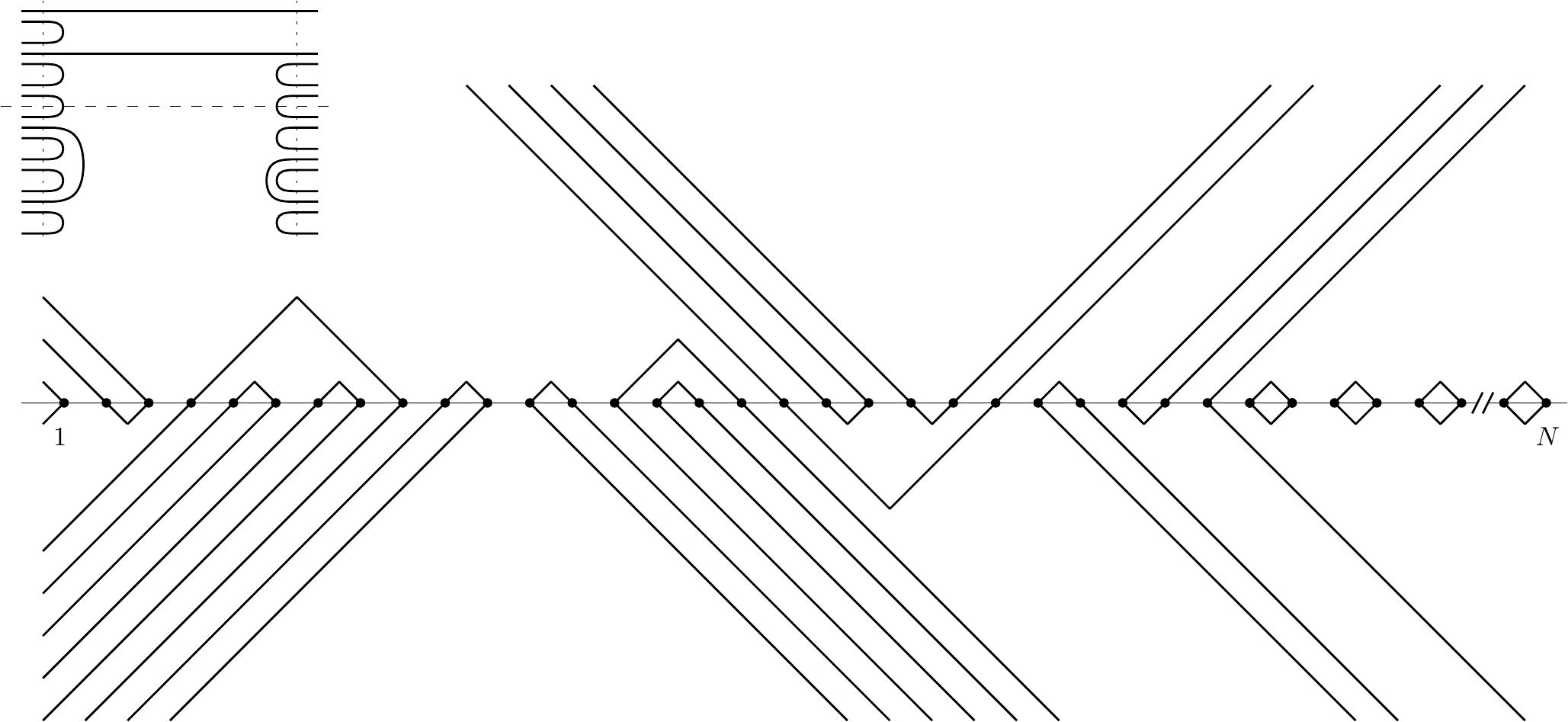}
\caption{Example of construction of local configuration realising a given matching (top left scheme), see Lemma~\ref{lem:local_config}. Here, $N$ has to be even and $\ge28$.}
\label{fig:exemple_matching}
\end{figure}

This construction uses a maximum total of $\ntot\le N$ vertices. Thus, this construction can be fitted inside $\Li1,N\Ri$ and leaves one free interval provided we used vertices that are next to each other. Furthermore, the number of used vertices is easily seen to have same parity as $n^+_L(\eta)+n^+_R(\eta)$, hence the remaining number of vertices has to be even due to the assumptions, which enables making a series of short loops $((+1,+1),(-1,-1))$ to complete the configuration in $\Li1,N\Ri$. 
\end{proof}

\begin{lemma}\label{lem:size_boundary}
Almost surely, the number of ends going out of $\Li1,N\Ri$ is negligible compared to $N$:
\begin{equation}\label{eqn:size_boundary}
\frac{\ntot(\omega_{|\Li1,N\Ri})}N\limites{}{N\to\infty}0\qquad \text{a.s.}
\end{equation}
\end{lemma}

\begin{proof} If for $k \geq 1$ we introduce $S_{k}^{+} = \omega^{+}_{1} + \cdots + \omega^{-}_{k}$ and similarly for $S^{-}$ then $S^{+}$ and $S^{-}$ are independent simple random walks and we have 
\[\ntot(\omega_{|\Li1,N\Ri})=\big(
\max_{1\le k\le N}S^+_k-\min_{1\le k\le N} S^+_k\big)+\big(
\max_{1\le k\le N}S^-_k-\min_{1\le k\le N} S^-_k)  \]
hence the conclusion comes from the law of large numbers.
\end{proof}

Finally, the following classical lemma controls probabilities after a local modification $\varphi$. 

\begin{lemma}[Finite energy property]\label{lem:finite_energy}
Let $S$ be a finite subset of $\Z$, and $\varphi$ be a mapping $\Omega_{S^c}\to\Omega_S$ (recall the notation from (\ref{def_omegaS})). For any event $C$, define
$\Ct=\{\omegat\st \omega\in C\}$,
where, for $\omega\in\Omega$, 
\[\omegat=\varphi(\omega_{|S^c})\indic_S+\omega_{|S^c}\indic_{S^c}.\] 
Then $\prob(\Ct)\ge 4^{-|S|}\prob(C)$. 
\end{lemma}

\begin{proof}
Note indeed that $C\subset\Cc$ where $\Cc=\{\omega\in\Omega\st \exists\eta\in\Omega_S,\ \eta\indic_S+\omega_{|S^c}\indic_{S^c}\in C\}$ is independent of $\omega_{|S}$, and $\omega_{|S}$ is uniform in $\Omega_S$, hence
$\prob(\Ct)=\prob(\Cc\cap\{\omega_{|S}=\varphi(\omega_{|S^c})\})=4^{-|S|}\prob(\Cc)\ge4^{-|S|}\prob(C)$.
\end{proof}

\section{Trifurcations}\label{sec:trifu}

In this section we prove that $\npath\le 2$ almost surely. The proof adapts Burton and Keane's classical argument based on an appropriate notion of multifurcation.  Let us first, for $x\in\Z$, consider the first infinite cluster straddling over $x$, if it exists:
\[\Cr^+(x)=\Cr(x'),\quad x'=\min\{y>x\st \exists z<x, \{y,z\}\in E^+,\,|\Cr(y)|=\infty\},\] 
and $\Cr^+(x)=\emptyset$ if there is no such cluster (it actually exists a.s.\ but we don't need this fact), and define similarly $\Cr^-(x)$ below $x$, using $E^-$.

In the following definition $d(x,A)$ is the distance in $ \mathbb{Z}$ of the point $x$ to a set $A \subset \mathbb{Z}$.
\begin{definition}A point $x\in\Z$ is called a \textbf{trifurcation point} if it belongs to the following set : 
\begin{align*}
 \Tri = \{x\in\Z \st 
 & (\omega^+_x,\omega^-_x)\in\{(-1,-1),(+1,+1)\},\,|\Cr(x)|=|\Cr^+(x)|=|\Cr^-(x)|=\infty,\\
 	& d(x;\Cr^+(x))\le3,\,d(x;\Cr^-(x))\le3,\\ 
	& \text{and } \Cr(x), \Cr^+(x), \Cr^-(x)\text{ are all different}\}
\end{align*}
\end{definition}
Notice that our definition is a bit more subtle than asking that $3$ infinite clusters come nearby $x$. Roughly speaking we do not want that one cluster separates the two-others. The need for our definition should become clearer in the proof of the following lemma.

\begin{lemma}\label{lem:trifurcation_point} For any $N$,
\[
\#(\Tri\cap\Li 1,N\Ri)\le 2+ \#\{\Cr:|\Cr|=\infty,\;\Cr\cap\Li 1,N\Ri\ne\emptyset\}.
\]
where $\Cr$ denotes a path of our graph $G$.

\end{lemma}

\begin{proof}
Let us view the paths $\Cr$ as embedded in $\R^2$ and compactify $\R^2$ by adding a point at infinity thus changing the bi-infinite lines into simple closed curves of the topological sphere $\R^2\cup\{\infty\}$ that are disjoint except for the point $\infty$. The proof revolves around an auxiliary planar graph $\Tr$, constructed as follows. The vertices of $\Tr$ are the faces of this embedding i.e.\ regions of $\R^2\cup\{\infty\}$ whose boundary is a union of paths $\Cr$, and two vertices are connected in $\Tr$ if they share a path $\Cr$ (infinite, intersecting $\Li 1,N\Ri$) as a boundary. Jordan's theorem and the fact that the paths intersect only at $\infty$ ensures that every path $\Cr$ is incident to exactly two faces, i.e.\ that $\Tr$ is indeed a graph.

We first note that $\Tr$ is a tree. This follows from the fact that its dual graph, which is a set of loops based on the same vertex, admits a spanning tree with one vertex and zero edges, and from~\cite[Theorem XI.6]{tut84} (the spanning trees of the dual of a planar graph $G$ are the complements of the edge-duals of spanning trees of $G$).

Let now $f$ be some vertex of $\Tr$ and let $\Ur$ be a collection of edges of $\Tr$ incident to $f$, or in other words, infinite paths contained in the boundary of the face $f$ (not necessarily all of them). Denote by $\Tri(\Ur)$ the set of $x\in\Tri$ such that $\{\Cr(x), \Cr^+(x),\Cr^-(x)\}\subset\Ur$. 

\begin{figure}
\includegraphics[height=4cm]{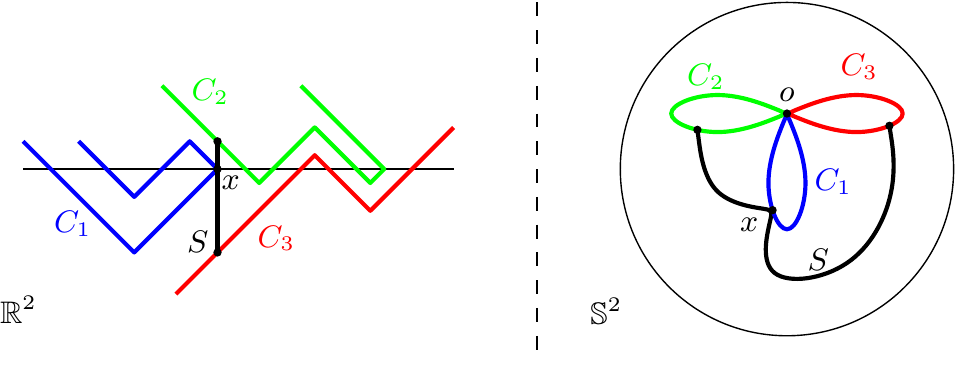}
\caption{Mapping of a trifurcation point into $\S^2$. The vertical segment $S$ meets $\Cr_1,\Cr_2$ and $\Cr_3$ only once. Any curve with that property must meet $S$, and thus have same abscissa $x$ if it is also a vertical segment. Cf.~proof of Lemma~\ref{lem:trifurcation_point}.}
\label{fig:topo_trif}
\end{figure}

We now claim that $|\Ur|\ge\Tri(\Ur)+2$.
We prove this by induction on $|\Ur|$. For $|\Ur|=0,1,2$ there is nothing to prove as a trifurcation requires $|\Ur|\ge 3$. Assume therefore that $|\Ur|\ge 3$ and that $x$ is some trifurcation point that belongs to $\Ur$. Consider the vertical segment $S$ in $\R^2$ passing through $x$ and whose upper and lower extremities are the first intersections with $\Cr^+$ and $\Cr^-$ respectively (here and below we write $\Cr^\pm$ as a short for $\Cr^\pm(x)$). See figure \ref{fig:topo_trif}. By the definition of $\Cr^\pm$, $S$ intersects no other infinite $\Cr$, in particular no other $\Cr\in\Ur$. Thus $S$ dissects the face $f$ into three disjoint parts. For example, one of them has as boundary the part of $\Cr^+$ up to the top of $S$, the part of $S$ in the top half plane, and then the part of $\Cr(x)$ up to $x$ (and possibly some paths in $\Ur\setminus\{\Cr(x),\Cr^+,\Cr^-\}$). Let $\Ur^{1,2,3}$ be the collection of paths of $\Ur$ which compose the boundaries of these three parts. By the construction, only $\Cr(x)$ and $\Cr^\pm$ may belong to more than one of $\Ur^{1,2,3}$ and each belongs to exactly two, so $|\Ur^1|+|\Ur^2|+|\Ur^3|=|\Ur|+3$. Further, each $y\in\Tri(\Ur)$, other than $x$ itself, must belong to one of the $\Tri(\Ur^i)$. Applying the claim inductively to $\Ur^i$ we get
\begin{align*}
  |\Tri(\Ur)|&\le 1+|\Tri(\Ur^1)|+|\Tri(\Ur^2)|+|\Tri(\Ur^3)|\\
    & \le 1+|\Ur^1|-2+|\Ur^2|-2+|\Ur^3|-2=|\Ur|-2
\end{align*}
as needed.

To conclude, let $e$ be the number of edges of $\Tr$, and let $n$ be the number of vertices. Since $\Tr$ is a tree $e=n-1$ and then
\[
e=2e-n+1=1+\sum_{f\in\Tr}(\deg(f)-1)\ge 1+\sum_{f\in\Tr}(\Tri(f)+1)
\]
where $\Tri(f)=\Tri(\Ur)$ for $\Ur$ the collection of all paths forming the boundary of $f$, and the last inequality is from the previous discussion. Since there is at least one $f$, the lemma is proved.
\end{proof}

Let us justify that trifurcation points may indeed occur if there are at least 3 infinite clusters. 

\begin{lemma}\label{lem:positive_trifurcation}
Assume $\npath\ge3$ almost surely. Then $\prob(0\in  \Tri)>0$. 
\end{lemma}

\begin{proof}
Assume $\npath\ge3$ almost surely. For $M\in\N$, let us consider the event $A(=A_M)$ that at least 3 infinite clusters meet $\Li-M,M\Ri$ and that the edge-boundary of $\Li-M,M\Ri$ has size at most $M/3$ in $G$, i.e.~$\ntot=\ntot(\omega |_{\Li-M,M\Ri})\le M/3$. The probability of this event goes to~1 as $M\to\infty$ due to Lemma~\ref{lem:size_boundary}, hence we may choose $M$ such that $P(A)>0$. In order to conclude, let us show that, on $A$, by changing the configuration inside $\Li-M,M\Ri$ appropriately, we can obtain a trifurcation point at~0 or at~1.

Recall that $\omega |_{\Li -M,M \Ri}$ encodes a perfect matching over $\Li 1, \ntot \Ri$, see Figure \ref{fig:ends_circle_in}. Similarly the outside configuration $\omega |_{ \mathbb{Z} \backslash \Li -M,M \Ri}$ encodes on $\Li 1 , \ntot \Ri$ another (planar) matching. However this matching is not anymore perfect since any infinite cluster of $G$ touching $\Li -M,M\Ri$ will be split into exactly two semi-infinite paths 
corresponding to two vertices of $\Li 1, \ntot \Ri$ that are unmatched (the breaking of the infinite path by restriciting to $\Z\setminus\Li -M,M\Ri$ might also create finite paths, but these correspond to pairs in the matching). On the event $A$, we thus have $2k$ points $1 \leq i_{1} < \cdots < i_{2k} \leq \ntot$ with $k \geq 3$ that correspond to the two semi-infinite lines coming for each cluster touching $\Li -M,M \Ri$. We first claim that $i_{2}-i_{1}-1, i_{3}-i_{2}-1, \ldots , i_{2k}-i_{2k-1}+1$ and $\ntot+i_{1}-i_{2k}+1$ are all \emph{even}. Indeed, inbetween two consecutive vertices $i_{\ell}$ and $i_{\ell+1}$ the exterior configuration $\omega |_{ \mathbb{Z} \backslash \Li -M,M\Ri}$ realises a planar perfect matching and so the number of vertices involved must be even.

We will now  modify the inside configuration $\omega |_{ \Li -M,M\Ri}$ in order to induce a perfect matching on $\Li 1, \ntot \Ri$ which pairs $i_{1} \leftrightarrow i_{2}$, $i_{3} \leftrightarrow i_{4}$ and $i_{5}\leftrightarrow i_{6}$ 
and all the other points to their neighbour. We want to do that in such a way that the three infinite clusters $ \mathcal{C}_{1}, \mathcal{C}_{2}$ and $ \mathcal{C}_{3}$ coming from the pairing of $i_{1},i_{2},\ldots, i_{6}$  have a trifurcation point either at $0$ or at $1$.

Depending on the number of ends of infinite paths on the left and right of $\Li-M,M\Ri$, we may reduce (up to horizontal symmetry) to one of the configurations in the top part of Figure~\ref{fig:pair_infinite}, where we only depict six ends of semi-infinite paths. 
\begin{figure}
\includegraphics[height=3cm,page=1]{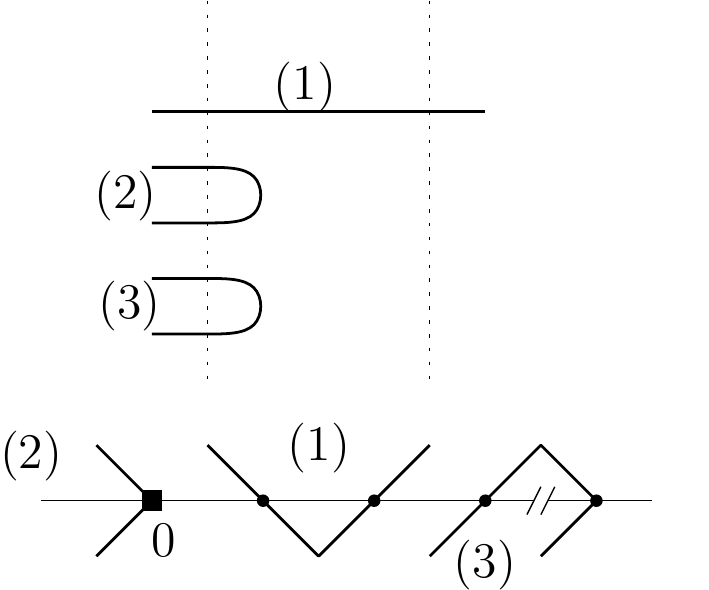}
\includegraphics[height=3cm,page=2]{trif_04.pdf}
\includegraphics[height=3cm,page=3]{trif_04.pdf}
\includegraphics[height=3cm,page=4]{trif_04.pdf}
\caption{Pairing the vertices $i_{1},i_{2}, \ldots, i_{6}$ and corresponding configurations around $0$ such that $0\in  \mathsf{Tri}$ (bottom). We don't specify here whether the vertices lie in the top or bottom half-plane (this is considered in the next figure). The five ``cut'' symbols on the bottom lines stand for an unspecified distance.
In the symmetric cases, applying a horizontal symmetry leads to trifurcation points where the configuration at $0$ is $(+1,+1)$.
}
\label{fig:pair_infinite}
\end{figure}

Depending on the location of the infinite paths, the bottom part of Figure~\ref{fig:pair_infinite} depicts how one sets the configuration around 0 (notice in figure the numbers which indicate which external edge connects to which local point) --- due a parity constraint, it may in fact be later necessary to shift the picture by 1, see below. Then one can accomodate for the finite clusters between the ends of infinite paths and for the other infinite clusters, in a way similar to Lemma~\ref{lem:local_config} except that it may be necessary to introduce turns to the semi-infinite paths (i.e.~to choose the configuration at some more vertices), depending whether each end lies on the top or bottom part, as sketched in Figure~\ref{fig:bending} in the case of two neighbouring infinite paths and readily extended to a larger number (up to four may be needed). 
\begin{figure}
\includegraphics[height=3cm,page=1]{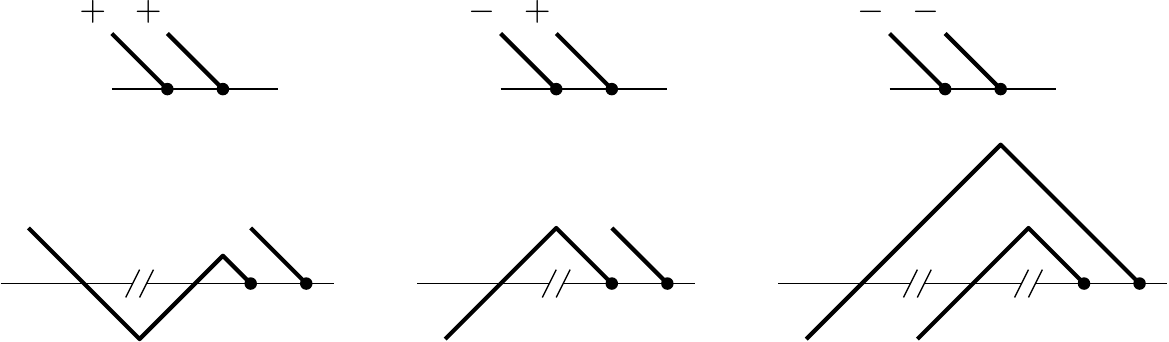}

\caption{Sketch of the ways to adapt the construction depending whether the ends of infinite paths lie in the top or bottom part. 
}
\label{fig:bending}
\end{figure}

In this way, we set so far the configuration at some vertices around 0 so that, disregarding the other vertices, the ends at the boundary of $\Li-M,M\Ri$ are connected as wished, and $0\in  \Tri$. In this construction, we could choose to use vertices around 0 without leaving empty space between them. 

This procedure uses a number of vertices that has same parity as $n^+_L+n^+_R$ (or equivalently as $n^-_L+n^-_R$). Since $\#\big(\Li-M,M\Ri\big)+n^+_L+n^+_R$ is even (cf.~Lemma~\ref{lem:local_config} above), the number of yet unsettled vertices inside $\Li-M,M\Ri$ is even. 
Up to a possible translation of the previously set configuration by 1, which would produce a trifurcation point at 1, we may thus assume that the numbers of unsettled vertices on the right and on the left of 0 are both even. We then complete the configuration in $\Li-M,M\Ri$ by length-two loops ($(+1,+1),(-1,-1)$) in the empty space. See Figure~\ref{fig:exemple_trifu} for an example.
\begin{figure}
\includegraphics[height=6cm]{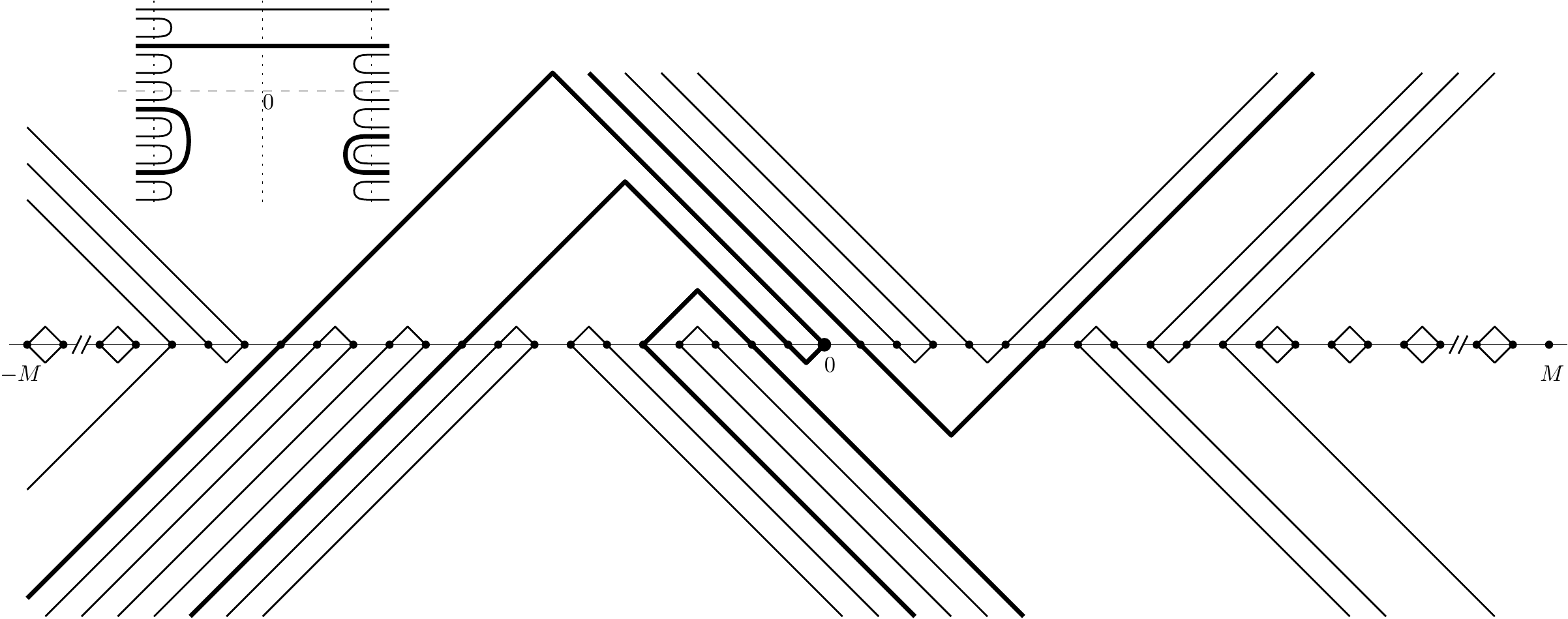}
\caption{Example of construction of a trifurcation by local modification. The top left scheme sketches the chosen matching of the ends of finite and infinite (thicker) paths. 
Since there are an even number of unsettled vertices on the left of 0 after the boundary ends have been matched, the trifurcation can be put at 0 and the configuration completed with short loops (the picture would otherwise have been shifted by 1).}
\label{fig:exemple_trifu}
\end{figure}

Matching two neighbouring ends of finite paths requires at most 2 vertices, and matching two ends of semi-infinite paths according to the previous rule requires at most 6 vertices. The fact that the number of ends is smaller that $M/{3}$ (from the definition of $A$) implies that there is indeed room for the construction.

Thus, on the event $A$, there is indeed a modification of the configuration within $\Li-M,M\Ri$ that leads to a configuration in $\{0\in \Tri\}\cup\{1\in \Tri\}$. Since $\prob(A)>0$, the finite energy property (Lemma~\ref{lem:finite_energy}) enables to conclude that $\prob(\{0\in \Tri\}\cup\{1\in \Tri\})>0$, hence $\prob(0\in \Tri)>0$ by translation invariance, which proves the lemma. 
\end{proof}

Let us now prove that $\npath \leq 2$ almost surely. Assume by contradiction that $\npath\ge3$ almost surely. Since the random subset $ \Tri$ is translation invariant and $\prob(0\in  \Tri)>0$, it follows from the ergodic theorem that 
\[\frac1n\#\big( \Tri\cap\Li1,n\Ri\big)\limites{}{n\to\infty}P(0\in  \Tri)>0\qquad\text{a.s.,}\]
where the lower bound is Lemma~\ref{lem:positive_trifurcation}. However, Lemma~\ref{lem:trifurcation_point} implies that, for all $n\in\N$, $\#\big( \Tri\cap\Li1,n\Ri\big)$ is larger than the number of infinite clusters that are involved in these trifurcations points; and, by the definition of $ \Tri$, these clusters all meet $\Li-2,n+3\Ri$, and therefore contribute to at least twice as many edges in the boundary of $\Li-2,n+3\Ri$ in $G$. In particular, the size of the boundary of $\Li-2,n+3\Ri$ in $G$ grows linearly in $n$. This contradicts Lemma~\ref{lem:size_boundary}.

\section{\texorpdfstring{$\npath \ne 2$.}{The number of paths is not 2.}}\label{sec:rewire}
In order to complete the proof of our main result, we have to rule out the possibility of $\npath =2$ which is the goal of this section.

\begin{proof}
Assume by contradiction that $\npath=2$ almost surely. Let us denote the two infinite paths by $\Cr_1$ and $\Cr_2$; for instance, we can choose indices so that $\Cr_1$ is closer to $0$ than $\Cr_2$ (equality cannot happen for parity reasons).

For any integers $k<l$, define the event $A_{\Li k,l\Ri}$ that both clusters meet $\Li k,l\Ri$:
\[A_{\Li k,l\Ri}=\{\Cr_1\cap\Li k,l\Ri\neq\emptyset,\, \Cr_2\cap\Li k,l\Ri\neq\emptyset\},\]
and for any integer $N$, let $B_N$ be the event that the boundary of $\Li1,N\Ri$ in $G$ has size smaller than $N$: 
\[B_N=\{\ntot(\omega_{|\Li1,N\Ri})\le N\}.\] 
Since $\prob(A_{\Li0,N\Ri})=\prob(A_{\Li-N/2,N/2\Ri})\uparrow 1$ and $\prob(B_N)\to1$ as $N\to\infty$ (by Lemma~\ref{lem:size_boundary}), we may take $N$ to be such that 
\[\prob(A_{\Li0,N\Ri})\ge \nicefrac 78
\qquad\text{and}\qquad \prob(B_N)\ge \nicefrac 78.\]
With $N$ established, let us pick $M$ such that
\[\prob(A_{\Li0,M\Ri})\ge 1-\frac14 4^{-N}.\]
For any integer $k<l$, let $R_{\Li k,l\Ri}$ be the maximum distance from $\Li k,l\Ri$ reached by a path in $G$ starting in $\Li k,l\Ri$ before coming back to $\Li k,l\Ri$ for the first time (and $R_{\Li k,l\Ri}=0$ if no path comes back). Since this distance is finite, there exists $r$ such that 
\[\prob(R_{\Li0,M\Ri}<r)\ge \nicefrac78.\] 
Let us finally define the event
\[C=A_{\Li0,N\Ri}\cap B_N\cap A_{\Li N+r,N+r+M\Ri}\cap\{R_{\Li 0,N\Ri}<r\}\cap \{R_{\Li N+r,N+r+M\Ri}<r\}.\]
and note already that
\begin{multline}\label{eqn:proba_c}
  \prob(C)\ge1-\prob(A_{\Li0,N\Ri}^c)-\prob(B_N^c)-2\prob(R_{\Li0,N\Ri}\ge r)-\prob(A_{\Li0,M\Ri}^c)\\
  \ge 1-\frac18-\frac18-\frac28-\frac144^{-N}>\frac14.
\end{multline}
Let us justify that, on the event $C$, it is possible to modify the configuration inside $\Li0,N\Ri$ in such a way that the event $A_{\Li N+r,N+r+M\Ri}$ is no longer satisfied. 

Assume $C$ holds. Due to $A_{\Li0,N\Ri}$, the infinite cluster $\Cr_1$ decomposes into one finite path that starts and ends in $\Li0,N\Ri$  and two semi-infinite paths that start in $\Li0,N\Ri$ and do not visit $\Li0,N\Ri$ again (by cutting the arbitrarily oriented path $\Cr_1$ at its very first and last visits in $\Li0,N\Ri$). Due to $R_{\Li 0,N\Ri}<r$ 
 the finite path does not visit $\Li N+r,N+r+M\Ri$ and due to $A_{\Li N+r,N+r+M\Ri}\cap\{R_{\Li N+r,N+r+M\Ri}<r\}$, exactly one of the semi-infinite paths does. 
The same holds for $\Cr_2$. We now claim that, by modifying the configuration in $\Li0,N\Ri$ it is possible to connect together the ends of these two semi-infinite paths starting in $\Li0,N\Ri$ and not visiting $\Li N+r,N+r+M\Ri$, without changing these semi-infinite paths outside $\Li0,N\Ri$; this results in $A_{\Li N+r,N+r+M\Ri}$ not holding anymore since one of the two infinite clusters is not visiting $\Li N+r,N+r+M\Ri$.

The possibility of such a modification comes from Lemma~\ref{lem:local_config}, whose applicability is granted by the event $B_N$ and the combination of the following two facts: 
\begin{itemize}
	\item due to planarity of the paths,
if exactly 4 semi-infinite paths exit $\Li0,N\Ri$, and only two of them meet $\Li N+r,N+r+M\Ri$, then these two can't separate the other two (denoted $x$ and $y$) in $\R^2\cup\{\infty\}$ 
and therefore, by planarity again, each of these pairs is separated by an even number of ends along the boundary of $\Li0,N\Ri$; 
	\item if two ends $x,y$ are separated by an even number of ends, then there exists clearly a non-crossing matching that matches $x$ to $y$. 
\end{itemize}

Therefore Lemma~\ref{lem:local_config} indeed applies to find a configuration that maintains the values of $n^\pm_L$ and $n^\pm_R$ and matches $x$ to $y$. 

\smallskip
Let us now conclude. Applying Lemma~\ref{lem:finite_energy} to the previously described modification, we get, by~\eqref{eqn:proba_c}, $\prob(\Ct)>\frac144^{-N}$. This contradicts the fact that $\Ct$ and $A_{\Li N+r,N+r+M\Ri}$ are disjoint, since that event has probability at least $1-\frac144^{-N}$. 
\end{proof}

\vspace{-.2cm} 

\section*{Acknowledgments}

We would like to thank Omer Angel for fruitful discussions in particular about parity arguments in the proof of the last section. 

\vspace{-.2cm} 

\bibliographystyle{siam}
\bibliography{bibli}
\end{document}

%% file: commandes.tex
\newcommand{\defeq}{\mathrel{\mathop:}=}

\newcommand{\R}{\mathbb{R}}

\renewcommand{\S}{\mathbb{S}}

\newcommand{\N}{\mathbb{N}}
\newcommand{\Z}{\mathbb{Z}}
\newcommand{\indic}{{\bf 1}}

\newcommand{\Cc}{\widehat C}

\newcommand{\Ct}{\widetilde C}

\newcommand{\omegat}{{\widetilde\omega}}

\newcommand{\Cr}{\mathcal C}

\newcommand{\Tr}{\mathcal T}
\newcommand{\Ur}{\mathcal U}

\newcommand{\limites}[2]{\overset{#1}{\underset{#2}{\longrightarrow}}}

\newcommand{\dm}{\begin{pmatrix}} 
\newcommand{\fm}{\end{pmatrix}}
\newcommand{\ddm}{\begin{vmatrix}} 
\newcommand{\fdm}{\end{vmatrix}}

\usepackage{textcomp,eurosym}

\newcommand{\st}{\,:\,} 